\documentclass[12pt,reqno]{amsart}

\usepackage[utf8]{inputenc}
\usepackage{amsthm}
\usepackage[foot]{amsaddr}
\usepackage{xcolor}
\usepackage{comment}
\usepackage{enumerate, fullpage}
\usepackage{listings}
\usepackage{comment}
\usepackage{stackengine}
\usepackage[pagebackref]{hyperref}
\usepackage{pgfplots}

\pgfplotsset{compat=1.16}
\pgfplotsset{
	width=0.75\textwidth,
	minor tick num=1
}

\newtheorem{definition}{Definition}
\newtheorem{lemma}{Lemma}
\newtheorem{theorem}{Theorem}
\newtheorem{remark}{Remark}
\newtheorem{conjecture}{Conjecture}
\newtheorem{corollary}{Corollary}

\author{Jon Kay}
\email{\href{mailto:jokay@ucsc.edu}{jokay@ucsc.edu}}
\address{Department of Mathematics. University of California, Santa Cruz.}

\author{Gerem\'ias Polanco}
\email{\href{mailto:gpolanco@smith.edu}{gpolanco@smith.edu}}
\address{Department of Mathematics. Smith College.}
\title{Relaxed Wythoff has All Beatty Solutions}
\date{\today}

\DeclareMathOperator*{\mex}{mex}

\newcommand{\goldenformula}{ {(2\!-\!t\!+\!\sqrt{t^2\!+\!4})}/2  }

\definecolor{yellow}{HTML}{F0E442}
\definecolor{green}{HTML}{009E73}
\definecolor{blue}{HTML}{0072B2}
\definecolor{red}{HTML}{D55E00}

\begin{document}

\begin{abstract}
We find conditions
under which the $P$-positions
of three subtraction games arise as pairs of complementary
Beatty sequences. The first game is due to Fraenkel
and the second is an extension of the first game to non-monotone
settings. We show that the $P$-positions of the second game can be inferred from
the recurrence of Fraenkel's paper if a certain inequality is
satisfied. This inequality is shown to be necessary if the $P$-positions are
known to be pairs of complementary Beatty
sequences, and the family
of irrationals for which this inequality holds is explicitly given.
We highlight several games in the literature that have $P$-positions as
pairs of complementary Beatty sequences with slope in this family.
The third game we present is novel, and we show
that the $P$-positions can be inferred
from the same recurrence in any setting.
It is shown that any pair of complementary Beatty sequences arises as 
the $P$-positions of some game in this family.
We also provide background on some inverse problems which have appeared
in the field over the last several years, in particular the Duch\^ene-Rigo
conjecture. This paper presents a solution
to the Fraenkel problem posed at the 2011 BIRS workshop, a modification of
the Duch\^ene-Rigo conjecture.
\end{abstract}

\maketitle


\section{Introduction}\label{introduction}

Wythoff's game starts with two piles, labeled $A$ and $B$, of finitely many 
tokens. Two players alternate, and at each turn the current player may either
\begin{enumerate}
\item[(a)] remove any positive number of tokens from a single pile, possibly
    the entire pile,
\item[(b)] or remove an equal amount of tokens from piles $A$ and $B$ simultaneously.
\end{enumerate} 
Under normal gameplay\footnote{In the literature, one also encounters mis\`ere 
gameplay under which the player who leaves both piles empty loses. Mis\`ere 
play is the typical setting of Nim, which we do  not adopt here.}, the player 
who leaves both piles empty wins. The game is itself a modification of Nim 
(under normal gameplay). In Nim, there may be more than two piles, and players 
may make only the first move (a) listed above. It is for this reason that the 
moves of type (a) are sometimes called \emph{Nim} moves. Moves of type (b) 
are sometimes called \emph{diagonal} moves or \emph{bishop} moves, as indicated 
by the representation of a game position by a piece on a chessboard 
\cite{wythoffwisdom} \cite{larsson_restrictions}.

If piles $A$ and $B$ contain $x$ and $y$ tokens, respectively, then label the 
game position as $(x,y)$. It is customary to make the assumption with no loss 
of generality that $x\leq y$. Label a position $(x,y)$ as an $N$-position if 
any player moving \emph{from} $(x,y)$ has a winning strategy. Otherwise, label 
$(x,y)$ as a $P$-position. Let $\mathcal{N}$ denote the set of $N$-positions 
and let $\mathcal{P}$ denote the set of $P$-positions. It is well known that 
the set of all game positions is partitioned by $\mathcal{P}$ and $\mathcal{N}$ 
because Wythoff's game is an acyclic combinatorial game with no ties 
\cite{siegel}.

Wythoff computed the $P$-positions of his game as the set
\begin{align}
\left\{\left(\left[ n\phi\right],
    \left[ n\phi^2\right]\right) \;|\; n\in\mathbb{Z}_{\geq 0} \right\}
\end{align}
where $[\cdot]$ denotes the floor function and $\phi={\left(1+\sqrt{5}\right)}/2$. 
Projecting onto the components of these pairs yields the sequences of integers 
$\{[n\phi]\}$ and $\{[n\phi^2]\}$ called the Beatty sequences of slope $\phi$ and
$\phi^2$, respectively. It is no 
coincidence that $1/\phi+1/\phi^2=1$, for Lord Rayleigh discovered in 1894 this 
condition precisely characterizes when the sequences $\{[n\phi]\}$ and
$\{[n\phi^2]\}$ exactly cover the positive integers \cite{schoenberg_1982}, a 
property which is intimately connected to the study of two-player
subtraction games. 
Beatty sequences have a rich literature that touches upon many topics in number 
theory and combinatorics (see for instance \cite{allouche_shallit_2003},
\cite{polanco_2017}, and the references therein). Formally they can be defined as
\begin{definition}[Homogeneous Beatty Sequence]
Let $\alpha>0$ be irrational. The sequence
\begin{align}
\{a_n\}_{n=0}^\infty = \left\{[n\alpha]\right\}_{n=0}^\infty
\end{align}
is called the \emph{(homogeneous) Beatty sequence} with slope $\alpha$.
In this paper, we are concerned only with homogeneous Beatty sequences,
so the first adjective will be omitted throughout.
\end{definition}
It is for the next theorem, mentioned in the previous paragraph, that we
may refer to a pair of \emph{complementary} Beatty sequences:
\begin{theorem}[Rayleigh-Beatty]\label{rayleigh}
Let $\alpha<\beta$ be positive irrationals and define
\begin{align}
A = \{[n\alpha]\}_{n=0}^\infty \quad B = \{[n\beta]\}_{n=0}^\infty.
\end{align}
Then $A\cup B=\mathbb{Z}_{\geq 0}$ and $A\cap B=\{0\}$ if and only if
$1/\alpha + 1/\beta=1$. In this case, $A$ and $B$ are said to be the
complementary Beatty sequencess parameterized by $\alpha$ or $\beta$.
\end{theorem}
In \cite{fraenkel_threefronts} and \cite{fraenkelborosh}, a countable family of Wythoff games is defined by adjoining valid moves. For a 
fixed positive integer $t$, the game $t$-Wythoff is played with
standard Nim moves and an extended diagonal move 
in which a player who removes $k$ and $\ell$ tokens from both piles
may do so if and only if $|k-\ell|<t$. The $P$-positions of games
in this family are given by the complementary Beatty sequences parameterized by $\alpha$ where
\begin{align}
\alpha = \frac{2-t+\sqrt{t^2+4}}{2}. \label{goldenfamily}
\end{align}
Note that $\alpha=\phi$ is the special case of $t=1$, which corresponds to the 
original Wythoff's game with golden $P$-positions. This family of
irrationals will appear repeatedly and its significance lies in the following.
\begin{lemma}\label{goldenfamilylemma}
If $\alpha=\goldenformula$ for some integer $t>0$ and $\beta$ satisfies
$1/\alpha+1/\beta=1$, then $\beta=\alpha+t$ and, in particular,
$\{\alpha\}=\{\beta\}$.
\end{lemma}

In the history of combinatorial games, it has always been attractive to
study the $P$-positions of a given game. This is the forward problem.
Lately, it has become attractive to study the \emph{inverse} problem:
where one starts with pairs from a given complementary sequence or family thereof and finds one or
more games which have pairs from those sequences as
$P$-positions. This has been discussed in \cite{fraenkel_new_2004},
\cite{fraenkel_gardner},
\cite{fraenkel_inverse},
\cite{fraenkelinvariance},
\cite{goldberg_extensions}, \cite{goldberg_rulesets}, and \cite{larssonweimer}.
A particular setting of this inverse problem was formalized by
Duch\^ene and Rigo in their conjecture, stated in
\cite{duchene_invariant_2010}.
\begin{conjecture}[Duch\^ene-Rigo]
Given a pair of complementary Beatty sequences $S = (A_n, B_n)_{n\geq 1}$,
there exists an invariant game having $S \cup {(0, 0)}$ as its set of
$P$-positions.
\end{conjecture}

This conjecture was proven by Larsson, Hegarty, and Fraenkel in
\cite{larsson_invariant_2011}. Though the existence is indeed satisfactory
for the resolution of the conjecture, it was noted in \cite{chromatic},
\cite{goldberg_extensions}, \cite{goldberg_rulesets}, and \cite{larssonweimer}
that the general solution is complicated for an ordinary
person to play.

\begin{figure}[b]
\begin{tikzpicture}
\begin{axis}[
xmin=0,xmax=100,
ymin=0,ymax=100,
width=0.4\textwidth,
height=0.4\textwidth,
xtick style={draw=none}, ytick style={draw=none}]
\addplot[only marks, mark options={scale=0.2}] table {data-1.txt};
\addplot[blue, domain=0:100] {(100/(pi+25))*x};
\addplot[blue, domain=0:100] {x/ (100/(pi+25)) };
\end{axis}
\end{tikzpicture}
\begin{tikzpicture}
\begin{axis}[
xmin=0,xmax=400,
ymin=0,ymax=400,
width=0.4\textwidth,
height=0.4\textwidth,
xtick style={draw=none}, ytick style={draw=none}]
\addplot[only marks, mark options={scale=0.1}] table {data-1.txt};
\addplot[blue, domain=0:400] {(100/(pi+25))*x};
\addplot[blue, domain=0:400] {x/ (100/(pi+25)) };
\end{axis}
\end{tikzpicture}
\caption{Legal moves in $(\mathrm{G}^*)^*$ depicted as ordered pairs.
Lines with slopes $1/(\beta-1)$ and $\beta-1$.
Adapted with an alternate parameter from ``Wythoff's Star''
so-named in \cite{larsson_combinartorial}.}
\label{positionsfigure}
\end{figure}
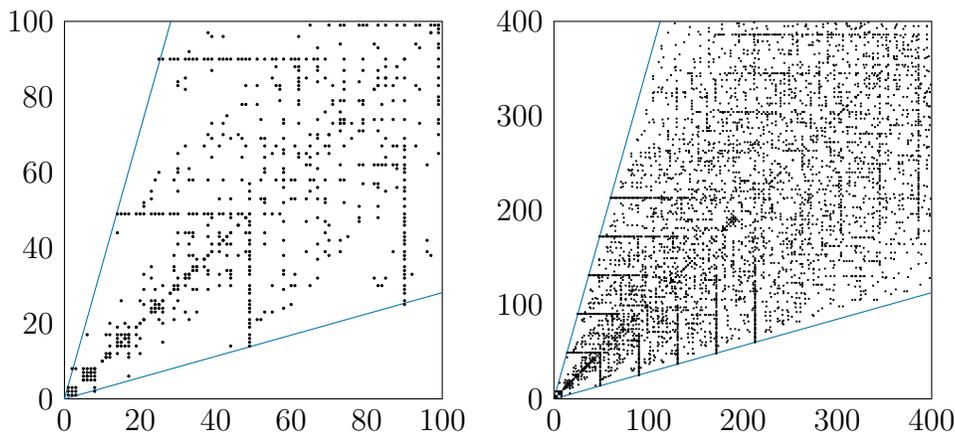
For a basic
explanation: start with desired complementary sequences; the $P$-positions
of an auxiliary game $\mathrm{G}^*$ must first be determined to obtain the
set of legal moves for a second game $(\mathrm{G}^*)^*$.
Only does this second game have $P$-positions
which are the desired complementary sequences. This means a player solving
the inverse problem accumulates additional information to obtain their
desired $P$-positions. Figure \ref{positionsfigure}
depicts valid moves for the second game when the known complementary sequences
are the complementary Beatty sequences with parameters $\alpha=5/4+\pi/100$
and $\beta$ satisfying $1/\alpha+1/\beta=1$. The complicated structure in the
figure indicates the player, in fact, accumulates much additional
information.

As such, an inverse problem without the requirement of invariance and
with the additional requirement of having
``nice'' rules has been formulated since then. Indeed, it was posed by Fraenkel
at the 2011 BIRS workshop in combinatorial games
\begin{quote}
``Find nice rules for a 2-player combinatorial game for which the $P$-positions
are obtained from a pair of complementary Beatty sequences,'' \cite{chromatic}
\end{quote}
In this paper, we are concerned with Fraenkel's
inverse problem, which has led to the
games of \cite{chromatic}, \cite{fraenkelinvariance},
\cite{larssonweimer} and others.

The main results of this paper are two-fold. First, an extension of the
game family of \cite{fraenkel_new_2004} is presented and its Beatty
$P$-positions are explicated, from which we recover an \emph{uncountable}
family of complementary Beatty sequences. Inspired by the derivation of the
previous result, a novel game family
is introduced with simple play rules
for which \emph{any} complementary Beatty sequences
are the $P$-positions of some game in the family, solving the inverse problem posed
by Fraenkel.

This paper is structured as follows. First a study of Definition \ref{general}
is made, determining necessary and sufficient condition under which games in 
this family  have $P$-positions which are complementary Beatty sequences.
Games in the 
family of Definition \ref{modifiedgame} are introduced and Theorem 
\ref{mastertheorem_expanded} indicates a sufficient criterion for when
one of these games 
has $P$-positions arising in terms of a distinguished recurrence formula. 
Theorem \ref{beattymodified2pilesaregolden} provides necessary and sufficient 
conditions so that a game in this family
has $P$-positions which are complementary Beatty sequences. Definition
\ref{relaxedgame} is presented and
the $P$-positions of games in this family are computed by Theorem 
\ref{relaxed_solutions}. Corollary \ref{main_result} states that any desired 
pair of complementary Beatty sequences arises as the $P$-positions of some
game in the family of Definition \ref{relaxedgame}.

\section{General 2-pile subtraction games}
In \cite{fraenkel_new_2004}, an 
infinite family of Wythoff's game variations is defined, again altering only the 
diagonal move. A game in this family is parameterized by a suitable function 
$f$ of the present and next game positions in the following sense: a player who 
removes $k$ and $\ell$ tokens from both piles may do so if and only if 
$|k-\ell|$ is bounded by $f$ (made precise in the following definition). 
Setting $f\equiv 1$ or $f\equiv t$ for some integer $t>1$ retrieve the special 
cases of the original Wythoff's game and $t$-Wythoff, respectively. Precisely,
the family is introduced as
\begin{definition}[General 2-pile subtraction games]\label{general} 
Given two piles of tokens $(x,y)$ of sizes $x,y$, with $0\leq x\leq y<\infty$. 
Two players alternate removing tokens from the piles:
\begin{enumerate}
\item[(aa)] Remove any positive number of tokens from a single pile, possibly
	the entire pile.
\item[(bb)] Remove a positive number of tokens from each pile by sending $(x_0,y_0)$
	to $(x_1,y_1)$ where
\begin{align}
|(y_0-y_1)-(x_0-x_1)| < f(x_1,y_1,x_0),
\end{align}
\end{enumerate}
where the constraint function $f(x_1,y_1,x_0)$ is integer-valued, positive, 
monotone, and semi-additive.
\end{definition}
The following theorem of the same paper determines the $P$-positions of any
game defined above.
\begin{theorem}[Fraenkel] \label{mastertheorem}
Let $\mathcal{S}=\{(a_n, b_n)\}_{n=0}^\infty$ where $a_0=b_0=0$ and
for all $n\in\mathbb{Z}_{>0}$,
\begin{align}
a_n &= \operatorname{mex}\{ a_k, b_k \;|\; k < n\}\\
b_n &= f(a_{n-1}, b_{n-1}, a_n) + b_{n-1} + a_n - a_{n-1}. \label{recurrence}
\end{align}
If $f$ is positive, monotone, and semi-additive, then $\mathcal{S}$ is the set 
of $P$-positions of a general 2-pile subtraction game with constraint function 
$f$.
\end{theorem}
One can see (proven in Lemma \ref{fraenkelcandidatelemma}) that the function
defined by
\begin{align}
f(x_1,y_1,[n\alpha]) = ([n\beta]-y_1)-([n\alpha]-x_1)
\label{fraenkelcandidate}
\end{align}
has $P$-positions which are pairs of the complementary Beatty sequences
parameterized by $\alpha$. However, this function is many-valued because, in
particular, it depends on the position moved to as well as from.
The next section gives a way to tell how many-valued the function is.

\section{Variance, Invariance, and $k$-Invariance}
There exists a classification of subtraction games according
how a player's allowed moves depend on the game state. In an
\emph{invariant} subtraction game, the move $(x,y)\to(x-k,y-\ell)$ is
allowed if and only
if for all positions $(x_0,y_0)$, the move
$(x_0,y_0)\to(x_0-k,y_0-\ell)$ is allowed (provided the pair is
non-negative in both components.) Subtraction games not satisfying this
property are called \emph{variant}. The following definition 
introduced in \cite{fraenkelinvariance} is weaker
than invariance:
\begin{definition}[$k$-Invariance]
A subtraction game is called $k$-invariant if its set of
positions can be partitioned into $k$ disjoint subsets such that, within
a subset, each allowed move is invariant in that subset.
 \end{definition}
The game played with the function in Equation \eqref{fraenkelcandidate} is not
$k$-invariant for any $k\in\mathbb{Z}_{>0}$, because the number of distinct values
of $f$ increases unboundedly. If the $P$-positions $\{(a_n,b_n)\}$ are known
and the conclusion of Theorem \ref{mastertheorem} holds,
one can see that a minimal requirement for such a function $f$ is that
\begin{align}
f(a_{n-1},b_{n-1},a_n)=(b_n-a_n)-(b_{n-1}-a_{n-1}).
\end{align}
As was observed in Example 3 of \cite{fraenkelinvariance},
this expression achieves at most three values when sequences $\{a_n\}$
and $\{b_n\}$ are complementary Beatty sequences. Therefore, if a constraint function depends
only of the position moved \emph{from}, then there is a chance to bound the
number of distinct values of $f$. For example, consider the following family satisfying
the minimal requirement:
\begin{definition}[Beatty constraint function]\label{beatty_constraint_function}
Let $1<\alpha<2$ be irrational. If $\beta$ satisfies $1/\alpha+1/\beta=1$, then
\begin{align}
f([n\alpha]) = f([n\beta]) = ([n\beta]-[(n-1)\beta])-([n\alpha]-[(n-1)\alpha]).
\label{modified_beatty_constraint_function}
\end{align}
is said to be the Beatty constraint function parameterized by $\alpha$.
\end{definition}
Since this expression achieves at most three values, we can deduce
that a Beatty constraint function induces a game which is $k$-invariant
for some $k\leq 3$.
For this reason, games played with these functions are the
object of central study in this paper.
The next lemma tells us when a Beatty constraint function may be used to play
a game of Definition \ref{general} by checking that it satisfies
the hypotheses required by that definition. When such a function may be used,
the $P$-positions are known
to be the complementary Beatty sequences parameterized by $\alpha$, according
to Theorem \ref{mastertheorem}.
\begin{lemma}\label{h_function_lemma}
Let irrationals $1<\alpha<2<\beta$ satisfy
$1/\alpha+1/\beta=1$. Define the sequence
\begin{align}
\label{deltasquare}
\Delta_{n-1}^2 = ([n\beta] - [(n-1)\beta]) - ([n\alpha]-[(n-1)\alpha]).
\end{align}
The following are equivalent:
\begin{enumerate}[(i)]
\item $\Delta^2$ is monotone
\item $\Delta^2$ is constant and equals $[\beta]-1$
\item $\alpha=\goldenformula$ where $t=[\beta]-1$.
\end{enumerate}
If one of the above is false, then the range of $\Delta^2$
lies within a subset of $\{[\beta]-2,[\beta]-1,[\beta]\}$.
\end{lemma}
\begin{proof}
The implication (ii) $\implies$ (i) is trivial, so let us prove
(i) $\implies$ (ii). Afterwards, we prove (ii) $\iff$ (iii). Suppose
the function is monotone.
Applying the decomposition $\beta=[\beta] + \{\beta\}$, we can rewrite
\begin{align}
[n\beta] &= [n([\beta]+\{\beta\})]\\
&= [n[\beta]+n\{\beta\}]\\
&= n[\beta]+[n\{\beta\}].
\end{align}
The first difference in parentheses in Equation \eqref{deltasquare} becomes
\begin{align}
[n\beta] - [(n-1)\beta] &= n[\beta] + [n\{\beta\}]
- ((n-1)[\beta]+[(n-1)\{\beta\}])\\
&= [\beta] + g_{1/\{\beta\}}(n-1).
\end{align}
where $g_{1/\rho}(n)=[(n+1)\rho]-[n\rho]$ for real $\rho$. After deriving a 
similar expression for the second difference in parentheses in Equation
\eqref{deltasquare} involving $\alpha$ we have,
\begin{align}
\Delta^2=[\beta]-[\alpha] + g_{1/\{\beta\}}(n-1) - g_{1/\{\alpha\}}(n-1)
= [\beta]-1 + d_\alpha(n-1), \label{reconstruction}
\end{align}
where the term
\begin{align}
d_\alpha(n) = g_{1/\{\beta\}}(n) - g_{1/\{\alpha\}}(n) \label{difference}
\end{align}
is the sole part of the expression which may vary. To determine precisely
 when $d_\alpha$ is monotone, start by defining the sets 
$X=\{[n/\{\alpha\}]\}_{n\geq 0}$ and $Y=\{[n/\{\beta\}]\}_{n\geq 0}$. From 
Chapter 9 of \cite{allouche_shallit_2003}, recall that
\begin{align}
g_{1/\{\alpha\}}(n) = \begin{cases}
1 \quad n \in X\\
0 \quad n \notin X
\end{cases} \qquad  g_{1/\{\beta\}}(n) = \begin{cases}
1 \quad n \in Y\\
0 \quad n \notin Y.
\end{cases}
\end{align}
Studying the possible cases, one obtains that $d_\alpha$
and $\Delta^2$ achieve at most three of the following values:
\begin{align}
d_\alpha(n) &=\! \begin{cases}
0 & n \in (X\cap Y)\!\cup\! (X^c\cap Y^c)\\
1 & n\in X^c\cap Y\\
-1 & n\in X\cap Y^c.
\end{cases} \!\!\implies \Delta^2_n \!=\! \begin{cases}
[\beta]-1\\
[\beta]\\
[\beta]-2.
\end{cases}
\label{trichotomy}
\end{align}
Theorem 3.11 of \cite{niven} states that if an
intersection of Beatty sequences is
non-empty and non-zero, then
the intersection in fact has infinite size.
This means if 
$d_\alpha$ achieves any value, then it achieves that value infinitely many 
times, indicating the function $d_\alpha$ oscillates if it achieves
more than one value. 
Therefore, if $d_\alpha$ is monotone, then it must achieve only one value. 
There are now three possibilities to examine, and we rule out the second two: 
that $d_\alpha\equiv 1$ or $d_\alpha\equiv -1$.

If $d_\alpha\equiv 1$, then $X^c\cap Y=\mathbb{Z}_{>0}$, so that $X=\emptyset$ 
and $Y=\mathbb{Z}_{\geq 0}$, a contradiction. A symmetric argument can be 
applied to eliminate the case $d_\alpha\equiv-1$, leaving possible only the 
remaining case $d_\alpha\equiv 0$, from which it follows that 
$\Delta^2_{n-1}\equiv[\beta]-1$.

To finish, a formula for $\alpha$ is derived, also providing a converse to
Lemma \ref{goldenfamilylemma}. 
The argument in the previous paragraph shows $(X\cap Y)\cup (X^c\cap Y^c)=\mathbb{Z}_{\ge 0}$. Since both $X$ and $Y$ are non-empty, this
implies $X=Y$. Proceed by
elementary algebra. Let $t=[\beta]-1$. Because 
$\{\alpha\}=\alpha-[\alpha]$ and $\{\beta\}=\beta-[\beta]$, it follows 
that
\begin{align}
\alpha - [\alpha] = \beta - [\beta]\text,
\end{align}
Substitute the values $\beta=\alpha/(\alpha-1)$, $[\alpha]=1$, and $[\beta]=t+1$
to obtain the equation
\begin{align}
\alpha - 1 &= \frac{\alpha}{\alpha-1} - (t+1)
\end{align}
which can be re-arranged into a quadratic equation.
The positive solution from the quadratic formula equals
\begin{align}
\alpha = \frac{2 - t + \sqrt{t^2+4}}{2},
\end{align}
completing the proof.
\end{proof}
\begin{remark}
The previous proof indicates that $\Delta_{n-1}^2\geq 0$,
for all $n$ because
\begin{align}
\Delta^2 &\geq [\beta]-2\geq 2-2=0.
\end{align}
\end{remark}
The preceding lemma and its proof constitute a refinement of Proposition 1
in \cite{fraenkelinvariance}, which studies the same difference,
by additionally finding the pre-image of each value
$[\beta]-2$,$[\beta]-1$, and $[\beta]$. The paper \cite{larssonweimer} also
makes heavy study of this difference in its Section 3.
The lemma (and the preceding remark) will be used in the next section also,
particularly in the proof of Theorem \ref{beattymodified2pilesaregolden}.
The preceding lemma also showed that a game played with a Beatty
constraint function $f$ is either invariant, 2-invariant, or 3-invariant.
In the latter cases, the game positions decompose into the disjoint union
\begin{align}
&\{(a_n,y), (b_n,y) \;|\;  n-1\in (X\cap Y)\cup(X^c\cap Y^c)\}\cup\\
&\{(a_n,y),(b_n,y) \;|\;  n-1\in X^c\cap Y\}\cup\\
&\{(a_n,y),(b_n,y) \;|\;  n-1\in X\cap Y^c\}.
\end{align}
where if the game is $2$-invariant, one of the preceding sets is empty.
In the next section, we expand Theorem \ref{mastertheorem} to be compatible
with non-monotone constraint functions. That some \emph{non-monotone} boolean
constraint function enjoys the conclusion of Theorem
\ref{mastertheorem} was observed in Section 4 of \cite{fraenkelboole}.
That the same conclusion is enjoyed by certain Beatty constraint functions
is what Section \ref{inverseproblem} is devoted to.

\section{Modified 2-pile subtraction games}
By the preceding discussion, it is of interest to consider a game
family including constraint functions lacking monotonicity.
Continuing to require $f\geq 1$
is mandated by a complementary setting, for if $f$ is vanishing, then
certain Wythoff moves are blocked entailing that $P$-positions may not be complementary.
This can be found in Table 6\footnote{This
example is presented by its author as a counterexample, but it appears to
instead be a genuine example. The main lemma of this section can be used to
prove this.} of \cite{fraenkel_new_2004}, where a vanishing constraint function
has the $P$-position $(1,1)$. By the proof of Lemma \ref{doublemex}, a game
of Definition \ref{modifiedgame} will have the $P$-position $(a_n,a_n)$ if and only
if $f(a_k,b_k,a_n)=0$ for all $k<n$.
If the constraint function depends only on the position moved from, the
latter condition is simply $f(a_n)=0$.

The main result of this section is a classification of the $P$-positions
of games in Definition \ref{modifiedgame}
which are pairs of complementary Beatty sequences. To remove the built-in
compatibility of Definition \ref{general} with Theorem \ref{mastertheorem},
we will work with the following game family, particularly with functions
in the family of Definition \ref{beatty_constraint_function}.
\begin{definition}[Modified 2-pile subtraction games]\label{modifiedgame}
Given two piles of tokens $(x,y)$ of sizes $x,y$, with $0\leq x\leq y<\infty$. 
Two players alternate, removing tokens from the piles as follows
\begin{enumerate}
\item[(aa)] Remove any positive number of tokens from a single pile,
    possibly the entire pile.
\item[(bb)] Remove a positive number of tokens from each pile by moving from
    $(x_0,y_0)$ to $(x_1,y_1)$ where
\begin{align}
|(y_0-y_1)-(x_0-x_1)| < f(x_1,y_1,x_0),
\end{align}
\end{enumerate}
where the constraint function $f$ is integer-valued.
\end{definition}
\begin{remark}\label{dependence}
Because we are interested in $k$-invariant games
(as discussed in the previous section), constraint functions
which depend only on the position moved
\emph{from} are prioritized. For this reason, it is understood (unless otherwise specified)
that a constraint function $f$ in this family satisfies $f(x_1,y_1,x_0)=f(x_0)$.
Setting the next lemma within this definition is convenient because
it permits two motivating examples at the end of the section.
\end{remark}
\begin{remark}
For the sake of consistent equation indexing, the following recurrence
starting with $(a_0,b_0)=(0,0)$ and proceeding by
\begin{align}
\begin{cases}
a_n = \mex\{a_k, b_k\;|\; k < n\}\\
b_n = f(a_{n-1},b_{n-1},a_n) + b_{n-1} + a_n-a_{n-1}
\end{cases} \text{for all $n>0$} \label{modified_recurrence}
\end{align}
is the recurrence of Equation \eqref{recurrence} adapted to the games
of Definition \ref{modifiedgame}.
\end{remark}

\begin{figure}[b]
{\small
\begin{tabular}{c|c|c|c|c|c|c|c|c|c|c}
$n$  &  0 & 1 & 2 & 3 & 4 & 5 & 6 & 7 & 8 & 9\\ \hline
$a_n$ & 0 & 1 & 2 & 4 & 7 & 8 & 9 & 10 & 11 & 12\\
$b_n$ & 0 & 3 & 6 & 5 & 13 & 16 & 19 & 15 & 23 & 26\\ \hline
$[n\alpha]$ & 0 & 1 & 2 & 4 & 5 & 7 & 8 & 10 & 11 & 13\\
$[n\beta]$ & 0 & 3 & 6 & 9 & 12 & 16 & 19 & 22 & 25 & 29
\end{tabular}
}
\caption{Top: $P$-positions of the game played in Definition \ref{modifiedgame}
with the constraint function in Equation 
\eqref{modified_beatty_constraint_function}
parameterized
by $\alpha=1+\sqrt{5}/5$. Bottom: Complementary Beatty sequences parameterized by $\alpha$.
Disagreement at $n=3$. }
\label{figure1}
\end{figure}

By playing a game in Definition \ref{modifiedgame} with the Beatty constraint
function parameterized by
$\alpha=1+\sqrt{5}/5$, one can see
the $P$-positions do not coincide with pairs of the 
recurrence in Equation \eqref{modified_recurrence} (see Figure \ref{figure1}
for a list of $P$-positions
and recurrence pairs written as Beatty pairs).
One could have decided that the failure of Theorem
\ref{mastertheorem} to apply is because the constraint function was
non-monotone, but this assumption is not necessary
(except that it was required by Definition 
\ref{general}).
Indeed, consider the non-monotone Beatty constraint function
with parameter $\alpha=1+(2\sqrt{19}-8)/2$. Playing a game with this
constraint function, the $P$-positions can be 
verified to coincide with pairs of the recurrence in Equation
\eqref{modified_recurrence}. See Figure \ref{figure2}
for several of the first $P$-positions.

\begin{figure}[t]
{\small
\begin{tabular}{c|c|c|c|c|c|c|c|c|c|c}
$n$  &  0 & 1 & 2 & 3 & 4 & 5 & 6 & 7 & 8 & 9\\ \hline
$a_n$ & 0 & 1 & 2 & 4 & 5 & 6 & 8 & 9 & 10 & 12\\
$b_n$ & 0 & 3 & 7 & 11 & 15 & 18 & 22 & 26 & 30 & 34\\ \hline
$[n\alpha]$ & 0 & 1 & 2 & 4 & 5 & 6 & 8 & 9 & 10 & 12\\
$[n\beta]$ & 0 & 3 & 7 & 11 & 15 & 18 & 22 & 26 & 30 & 34
\end{tabular}
}
\caption{Top: $P$-positions of the game played in
Definition \ref{modifiedgame}
with the constraint function in Equation
\eqref{modified_beatty_constraint_function} parameterized
by $\alpha=1 + (2\sqrt{19}-8)/2$. Bottom: Complementary
Beatty sequences parameterized by $\alpha$.}
\label{figure2}
\end{figure}

In the next lemma, we solve the forward problem of finding the $P$-positions
of a particular game in the family of Definition \ref{modifiedgame},
allowing us to designate our specific \emph{inverse problem} in the next
section.

Now that $f$ has been assumed to be merely integer-valued, the next lemma
is a generalization of Theorem \ref{mastertheorem}.
Algorithms like the one exposed in the statement and
proof of this lemma are know as Minimum
Excluded algorithms and variations of this proof strategy have their origins
in Wythoff's original paper 
 \cite{wythoffwisdom} \cite{wythoff}. Contemporary references include 
\cite{duchene_invariant_2010}, \cite{fraenkel_threefronts},
\cite{fraenkel_new_2004}, \cite{larsson_restrictions}, and many others.
Later on, we adapt the lemma to analyze 
the game in Definition \ref{relaxedgame}, exposing its $P$-positions
in Theorem \ref{relaxed_solutions}.

The statement of this lemma shows how to compute the set of $P$-positions.
The proof shows that a player in a $P$-position may never move to
a lower $P$-position, and that a player in an $N$-position may always move
to a $P$-position, so that we may apply the $P$- and $N$-position structure
theorem.
\begin{lemma}\label{doublemex}
Let $\mathcal{S}=\{(a_n,b_n)\}_{n=0}^\infty$ where $a_0=b_0=0$
and for all $n \in \mathbb{Z}_{>0}$
\begin{align}
a_n &= \mex\{a_k, b_k \;|\; k < n\}\\
b_n &= \min\{b \geq a_n \;|\; b\!\neq\!b_k \;\text{and}\;
    |(b\!-\!b_k)\!-\!(a_n\!\!-\!a_k)|  \!\geq\! f(a_k,b_k,a_n)\forall k\!<\!n
    \}\label{desiredb}.
\end{align}
Then $\mathcal{S} = \{(a_n, b_n)\}_{n=0}^\infty$ is the set of
$P$-positions of a game in Definition \ref{modifiedgame} with constraint function $f$.

\end{lemma}
\begin{remark}
The minimum above can be rewritten as the mex:
\begin{align}
b_n &= \mex_{\geq a_n}
\left(\{b_k\}_{0<k<n}
    \cup\{b\geq a_n \;|\;
    |(b\!-\!b_k)\!-\!(a_n\!\!-\!a_k)|  \!<\! f(a_k,b_k,a_n)
    \}_{k<n} \right)\label{desiredb_mex}.
\end{align}
\end{remark}
\begin{proof}
In the setting of a subtraction game ending at $(0,0)$
the $P$-positions can be generated inductively, starting with $(0,0)$.
Label the first non-zero $P$-position $(a_1,b_1)$.
The component $a_1=\mex\{0\}$ is indicated because $(0,b)$ can be moved
to $(0,0)$ for any $b>0$.
The component $b_1$ is found by taking the minimum $b$ for which
$(a_1,b)$ has no move to lower $P$-positions:
\begin{align}
b_1 &= \min\{b\geq 1 \;|\; |(b-0)-(1-0)| \geq f(0,0,1)\}\\
&= 1 + f(0,0,1).
\end{align}
At the $n$th step, the position $(a_n,b_n)$ is computed similarly by taking
a mex for the first component
\begin{align}
a_n = \mex\{a_k,b_k \;|\; k < n\}.
\end{align}
If $a_n$ were not equal to this mex, then $a_n=a_k$ or $a_n=b_k$ for some
$k<n$.
In the first case, $(a_n,b_n)=(a_k,b_n)$. If $b_n>b_k$, then moving to
$(a_k,b_k)$ is possible by a Nim move, a contradiction.
Otherwise, if $b_n<b_k$, then the $P$-position
$(a_k,b_k)$ has a move to the $P$-position $(a_n,b_n)$ via a Nim move, which
is also a contradiction. Finally, if $a_n=b_k$, then the $P$-position
$(a_n,b_n)=(b_k,b_n)$ has a move to $(b_k,a_k)$ which is relabelled as the
$P$-position $(a_k,b_k)$, also a contradiction.

For the second component, now that $n>1$, we need to add a quantifier and restrict
$b\neq b_k$:
\begin{align}
b_n \!=\! \min\{b\!\geq\! a_n \;|\; b\neq b_k \;\text{and}\;
	|(b\!-\!b_k)-(a_n\!-\!a_k)| \!\geq\! f(a_k,b_k,a_n) \forall k\!<\!n\}.
\end{align}
Note that we can be sure $b_k\geq a_k$ because that is the convention we 
adopted for denoting game positions. Moreover, $b_n=a_n$ if and only if
$f(a_k,b_k,a_n)=0$ for all $k<n$.

The above paragraph indicates how to compute the set $\mathcal{S}$. Now we show 
that if a player is in any game position in the
complement of $\mathcal{S}$, then they can 
always move directly to a game position in $\mathcal{S}$.
Once this is proven, the structure theorem of $P$- and $N$-positions
implies that $\mathcal{S}$ is the set of $P$-positions.

It suffices to 
consider the following game positions, because the sequence $\{a_n\}$ consists 
of minimum excludants:
\begin{align}
(a_n, a_m) \quad (a_n, b_m) \quad (b_n, a_m) \quad (b_n, b_m)
\end{align}
for distinct $m$ and $n$.
If $f(a_k,b_k,a_n)=0$ for all $k<n$, then $b_n=a_n$ and only the first two
cases need to be inspected. For the first position, note that $a_m>a_n=b_n$,
so that moving to $(a_n,b_n)$ is possible. For the second position,
note that $b_m>a_n=b_n$, so a player can move to $(a_n,b_n)$.

Otherwise, suppose $f(a_k,b_k,a_n)\neq 0$ for some $k<n$.
To handle the first case, suppose $a_m>b_n$. Then a player may move
to $(a_n,b_n)$. Otherwise we have $a_m<b_n$. Interpreting $b_n$
as the minimizer
\begin{align}
b_n &= \min\{b \geq a_n \;|\; b\!\neq\!b_k \;\text{and}\;
|(b\!-\!b_k)\!-\!(a_n\!\!-\!a_k)|  \!\geq\! f(a_k,b_k,a_n)\forall k\!<\!n
\},
\end{align}
one can see that there exists $k<n$ such that
\begin{align}
|(a_m - b_k) - (a_n - a_k)| < f(a_k,b_k,a_n)
\end{align}
which indicates a player may make the diagonal move to $(a_k,b_k)$. For the 
second case, if $b_m > b_n$, then a player may make a Nim move to $(a_n,b_n)$. 
If $b_m < b_n$, then interpreting $b_n$ as a minimizer again reveals that a 
player may make a diagonal move from $(a_n, b_m)$ to some $(a_k,b_k)$.

For the third case, if $n<m$ then a player may make a Nim move to $(b_n,a_n)$ 
which is relabeled as $(a_n,b_n)$. Otherwise we shall have $m < n$, thus
obtaining the strict inequality
\begin{align}
a_m < a_n \leq b_n \leq a_m,
\end{align}
which is a contradiction, because the case $a_n=b_n$ was already addressed.
Finally, suppose a player is in the fourth position-type above.
If $b_m > a_n$, then a player can make a Nim move 
to $(b_n, a_n)$. Otherwise, $b_m < a_n$ implies the strict inequality
\begin{align}
b_n \leq b_m < a_n \leq b_n,
\end{align}
but we already handled the case $a_n=b_n$.

Therefore, $\mathcal{S}$ is the set of $P$-positions for a game played
with the constraint function $f$.
\end{proof}

The next two examples show how the previous lemma can be used to deduce
$P$-positions. First, let us prove the observation surrounding
Equation \eqref{fraenkelcandidate}.
\begin{lemma}\label{fraenkelcandidatelemma}
If a game in Definition \ref{modifiedgame} is played with the constraint
function
\begin{align}
f(x_1,y_1,[n\alpha]) = ([n\beta]-y_1)-([n\alpha]-x_1)
\end{align}
then the $P$-positions are pairs of the complementary Beatty sequences
parameterized by $\alpha$.
\end{lemma}
\begin{proof}
By induction, suppose $(a_k,b_k)=([k\alpha],[k\beta])$ for all
$k<n-1$. This is true in the base case $(0,0)$.
At the $n$th step, the $P$-position $(a_n,b_n)$ is computed by
Lemma \ref{doublemex} as follows:
\begin{align}
a_n &= \mex\{a_k,b_k\;|\; k<n\}\\
b_n &= \mex(\{b_k\} \cup [a_n + b_k-a_k - f(a_k,b_k,a_n) + 1,
    a_n + b_k-a_k + f(a_k,b_k,a_n)-1]). \label{fraenkel_b_nmex}
\end{align}
The upper endpoint of any of the intervals in Equation \eqref{fraenkel_b_nmex}
can be simplified as
\begin{align}
a_n + b_k-a_k + f(a_k,b_k,a_n)-1 &= a_n + b_k-a_k + [n\beta]-a_n - (b_k-a_k)-1\\
&= [n\beta]-1.
\end{align}
The lower endpoint the intervals can be simplified as
\begin{align}
a_n + b_k-a_k - f(a_k,b_k,a_n)+1 &= a_n + b_k-a_k - [n\beta]+a_n + (b_k-a_k)+1\\
&= 2a_n - [n\beta] + 2(b_k-a_k) + 1.
\end{align}
This indicates the intervals expand. After taking the entire union, we are
left evaluating $\mex_{\geq a_n}([2a_n-[n\beta]+1, [n\beta]-1])$.
Since $[n\beta]>a_n$, we can see that $2a_n-[n\beta]<a_n$, and so
$2a_n-[n\beta]+1\leq a_n$, which indicates the $\mex$ equals $[n\beta]$.
\end{proof}

Let us also prove that the third counterexample in Proposition 2 of
\cite{fraenkel_new_2004} is a actually a genuine example. The constraint
function in question is
\begin{align}
f(x_1,y_1,x_0)=(1+(-1)^{y_1+1})x_1/2.\label{fraenkel_nonexample}
\end{align}
Table 6 of the same paper contains several
$\mathcal{S}$-positions according to the recurrence formula of Equation
\eqref{recurrence}. The author claims these are not $P$-positions, because
``$(10,29)$ cannot be moved to any of the lower $\mathcal{S}$-positions'',
suggesting that the pair $(10,31)$ from the table is not a $P$-position,
because $(10,29)$ ought to be.
However, a closer look shows that moving from $(10,29)$ to the
$\mathcal{S}$-position $(8,21)$ \emph{is} allowed,
because $|(29\!-\!21)-(10\!-\!8)|=6<f(8,21,10)=(1+(-1)^{22})8/2=8$.
The remaining $P$-positions can be computed using Lemma \ref{doublemex},
showing the same recurrence formula holds.
\begin{lemma}
If a game in Definition \ref{modifiedgame} is played with the constraint function
in Equation \eqref{fraenkel_nonexample},
the set of $\mathcal{S}$-positions starting with $(a_0,b_0)=(0,0)$
and $(a_1,b_1)=(1,1)$, followed by
\begin{align}
a_n &= \mex\{a_k, b_k \;|\; k < n\}\\
b_n &= a_n + b_{n-1} - a_{n-1} + f(a_{n-1},b_{n-1},a_n)\\
&= \begin{cases}
a_n + b_{n-1} & \text{$b_{n-1}$ odd}\\
a_n + b_{n-1}-a_{n-1} & \text{$b_{n-1}$ even}.
\end{cases}
\end{align}
for all $n\geq 2$ are $P$-positions.
\end{lemma}
\begin{proof}
Base cases at $n=0,1$ can be verified by hand.
For an induction step, the next $P$-position is found by
applying Lemma \ref{doublemex}
\begin{align}
a_n &= \mex\{a_k, b_k \;|\; k < n\}\\
b_n &= \mex(\{b_k\} \cup [a_n + b_k - a_k - f(a_k, b_k) + 1,
                        a_n + b_k - a_k + f(a_k, b_k) - 1]_{k<n}). \label{fraenkel_intervals}
\end{align}
Note that the intervals
\begin{align}
[a_{n-1} + b_k-a_k-f(a_k,b_k)+1,
a_{n-1} + b_k-a_k+f(a_k,b_k)-1]_{k<n-1}
\end{align}
cover the set $\{a_{n-1},a_{n-1}+1,\dots,b_{n-1}-1\}$ by the inductive hypothesis.
Therefore it can be seen that the shifted intervals
\begin{align}
[a_n + b_k-a_k-f(a_k,b_k)+1,
a_n + b_k-a_k+f(a_k,b_k)-1]_{k<n-1}
\end{align}
cover the set $\{a_n,a_n +1,\dots,b_{n-1}-1+(a_n-a_{n-1})\}$.

If $b_{n-1}$ is even, then $f(a_{n-1},b_{n-1})=0$ and the last interval of Equation
\eqref{fraenkel_intervals} is
empty, so we simply take $b_n=a_n+b_{n-1}-a_{n-1}$.
If $b_{n-1}$ is odd, then $f(a_{n-1},b_{n-1})=a_{n-1}$ and the last interval considered is
\begin{align}
[a_n + b_{n-1} - 2a_{n-1} + 1, a_n + b_{n-1} - 1].
\end{align}
There are no integers in-between the set $\{a_n,a_n +1,\dots,b_{n-1}-1+(a_n-a_{n-1})\}$
and this interval, because we can see that
\begin{align}
a_n + b_{n-1} - 2a_{n-1} + 1 &\leq b_{n-1}-1+(a_n-a_{n-1})+1 \iff \\
1 &\leq a_{n-1}
\end{align}
which holds for all $n\geq 2$. Therefore, we can determine that
$b_n = a_n + b_{n-1}$.
\end{proof}
Now that we have seen two examples of how Lemma \ref{doublemex} can be used to
determine a game's $P$-positions, let us move on to the main result characterizing
when the $P$-positions in games of Definition \ref{modifiedgame} are pairs of
complementary Beatty sequences.

\section{The Inverse Problem}\label{inverseproblem}
In this section, we pose our inverse problem:
if the $P$-positions of a game in Definition \ref{modifiedgame}
are pairs of complementary Beatty sequences, what is the constraint function of that
game? Is it the Beatty constraint function? From which functions
arise games whose $P$-positions are pairs of
complementary Beatty sequences? These three questions are answered in this section,
per the proviso in Remark \ref{dependence}.

The next theorem refines the minimum for a
class of constraint functions which satisfy a certain inequality.
\begin{theorem}\label{mastertheorem_expanded}
Let $f\geq 1$ be a function of integers into integers and let $n>1$ be fixed. Suppose a game 
in Definition \ref{modifiedgame} has $P$-positions $(a_k,b_k)$ which satisfy
\begin{align}
\label{eqforbsubk}
a_k &= \mex\{a_j, b_j \;|\; j < k\}\\\nonumber
b_k &= f(a_k) + b_{k-1} + a_k - a_{k-1}
\end{align}
for all $0<k < n$. If $2f(a_n)-f(a_k)\geq 1$ is true for all $0<k<n$,
then
\begin{align}
b_n = f(a_n) + b_{n-1} + a_n - a_{n-1}. \label{b_nformula}
\end{align}
\end{theorem}
\begin{proof}
First off, if $f(a_n)=1$, we can deduce the inequality
\begin{align}
2 - f(a_k) \geq 1 \implies 1 \geq f(a_k) \implies f(a_k)=1 \forall k<n
\end{align}
from which we obtain the classical Wythoff game. The recurrence formula is
known in this case. Otherwise, suppose $f(a_n)\not\equiv 1$.

By combining the second line of Equation \eqref{eqforbsubk} and the
hypothesis $f\ge 1$, one can deduce the sequence $\{b_k-a_k\}$ is
monotone, because
\begin{align}
b_k-a_k &= b_{k-1}-a_{k-1} + f(a_k)\\
&\geq b_{k-1}-a_{k-1} + 1
\end{align}
As in the remark following Lemma \ref{doublemex}, we can
re-write the minimum for $b_n$ of the same lemma as a minimum excludant:
\begin{align}
\label{b_nmex}
b_n &= \mex_{\geq a_n} \left( \{b_k\}_{k<n}\cup
	\{b \geq a_n \;|\; |(b\!-\!b_k)\!-\!(a_n\!\!-\!a_k)| < f(a_n)\}_{k<n}\right).
\end{align}
By re-writing the inequality $|(b\!-\!b_k)\!-\!(a_n\!\!-\!a_k)| < f(a_n)$ we conclude that $b$ belongs to the interval
of integers
\begin{align}
I_k = [a_n\!+\!b_k\!-\!a_k\! -\! f(a_n)\! +\! 1,
	a_n\!+\!b_k\!-\!a_k\! +\! f(a_n)\! -\! 1] \cap \mathbb{Z}_{\geq a_n}.
    \label{interval_definition}
\end{align}
Therefore, the mex in equation \eqref{b_nmex} is an integer lying outside any interval $I_k$. With
the assumption $2f(a_n) - f(a_k) \geq 1$ for
all $0<k<n$, we can deduce that the union of
the intervals $\{I_k\}$ is also an interval. To
see this, recall that $\{b_k-a_k\}$ is monotone.
This indicates that the intervals can be ordered
monotonically also. Now, count the number of
integers lying between the right-endpoint of
the interval $I_k$ and the left-endpoint of the
interval $I_{k+1}$:
\begin{align}
(a_n + b_{k+1}-a_{k+1} - f(a_n))
- &(a_n + b_{k+1}-a_{k+1} + f(a_n)) + 1\\
&= f(a_{k+1}) - 2f(a_n) + 1\\
&\leq 0
\end{align}
where the last inequality follows from the
assumption. Thus the number of integers between consecutive intervals is
less than or equal to zero. In other words, we arrive at the \emph{important conclusion}
that there is no gaps between interval, i.e., consecutive intervals are either
overlapping or the starting point of the next interval is exactly one integer
away from the end point of the previous interval. This permits an ansatz for
the minimum excluded $b_n$ by simply evaluating
$\max I_n+1$.
\begin{align}
\widetilde{b_n} = a_n + b_{n-1}-a_{n-1} + f(a_n).
\end{align}
Now that we have a candidate for $b_n$, we need to just be sure that 
$\widetilde{b_n}\neq b_j$ for all $j<n$, which can be demonstrated by 
contradiction. Suppose
\begin{align}
b_j = a_n + b_{n-1} - a_{n-1} + f(a_n)
\end{align}
for some $j<n$.
Since the sequence $\{a_k\}$ is monotone and $j<n$ we can deduce that
\begin{align}
&b_j > a_j + b_{n-1} - a_{n-1} + f(a_n)\\
&\implies (b_j - a_j) - (b_{n-1} - a_{n-1}) > f(a_n).
\end{align}
But since $\{b_k-a_k\}_{k<n}$ is monotone, this implies that $f(a_n)<0$, 
contradictory to our assumption that $f\ge 1$.
\end{proof}

Applying the previous result, we can deduce
the following special case.
\begin{lemma}\label{trivialnogaplemma}
Let a game in Definition \ref{modifiedgame} be played with the Beatty
constraint function parameterized by some $\alpha<5/4$. Then the
$P$-positions are pairs of the complementary Beatty sequences
parameterized by $\alpha$.
\end{lemma}
\begin{proof}
The assumption $\alpha<5/4$ and the relation
$\alpha^{-1}+\beta^{-1}=1$ imply that $[\beta]\geq 5$.
The study of the difference function in Lemma
\ref{h_function_lemma} indicates
that the constraint function
may take values only from the set
$\{[\beta]-1,[\beta],[\beta]-2\}$
which implies the constraint
function is bounded from below,
\emph{i.e.}, $f\ge3$. Suppose the function $f$ is three-valued. Then
\begin{align}
2\min f - \max f = 2([\beta]-2)-[\beta] = [\beta]-4 \geq 1.
\end{align}
If $f$ is two-valued, then there are two other possibilities to exhaust:
\begin{align}
2\min f - \max f &= 2([\beta]-1)-[\beta] = [\beta]-2 \geq 3\\
2\min f - \max f &= 2([\beta]-2)-([\beta]-1) = [\beta]-3 \geq 2.
\end{align}
If $f$ is constant, $2\min f-\max f\equiv f\geq 3$.
In all of these cases, we see that $2\min f - \max f \geq 1$, so that
Theorem \ref{mastertheorem_expanded} may be applied. To explicate that the $P$-positions
are indeed complementary Beatty sequences,
an inductive argument is sufficient.
Assume that $(a_n,b_n)=([n\alpha],
[n\beta])$ for all $n<N$. This is true for
$N=1$. For the inductive step, note that
the formula of Equation \ref{b_nformula} holds
for all $n>0$. We can thus deduce that
\begin{align}
b_N &= f(a_N) + b_{N-1}+a_N-a_{N-1}\\
&= ([N\beta]-[N\alpha])
    -([(N-1)\beta]-[(N-1)\alpha])
+ b_{N-1} + a_N-a_{N-1}\\
&= ([N\beta]-[N\alpha])
    -([(N-1)\beta]-[(N-1)\alpha])
+ [(N-1)\beta] + a_N-[(N-1)\alpha]\\
&= [N\beta] + a_N-[N\alpha]
\end{align}
where the inductive hypothesis was applied
to substitute values for $b_{N-1}$ and
$a_{N-1}$.
It is a property of complementary Beatty
sequences that $a_N=\mex\{a_k,b_k\}=[N\alpha]$,
which then implies $b_N=[N\beta]$, completing
the induction.
\end{proof}

The next lemma demonstrates that the inequality
derived in the proof of Theorem \ref{mastertheorem_expanded} is a
\emph{necessary} property of a constraint function whose game has
$P$-positions which are complementary Beatty sequences.
\begin{lemma}
Suppose a game in Definition \ref{modifiedgame} is played with the
Beatty constraint function $f$ parameterized by $\alpha$. Then the
$P$-positions
are the complementary Beatty sequences
if and only if $2\min f - \max f \geq 1$.
\end{lemma}
\begin{proof}
Theorem \ref{mastertheorem_expanded} shows that
$2\min f - \max f \geq 1$ implies the $P$-positions are pairs of the
complementary Beatty sequences parameterized by $\alpha$ after expanding
the recurrence.

For the reverse
direction, we prove its contrapositive, that
$2f\min f - \max f < 1$ implies the $P$-positions
are not complementary Beatty sequences.
Suppose $2\min f - \max f < 1$. Continue as in the
proof Theorem \ref{mastertheorem_expanded}, where a sequence of
intervals is considered. Select $n>1$ and
$0<k<n$ such that there is a gap between the endpoints of the
intervals $I_{k-1}$ and $I_k$ of maximal size.
The size of the gap equals $f(a_k)-2f(a_n)+1$. If this is
greater than one, then we can
be sure the gap is not filled by the sequence $\{b_k\}$, because this
sequence grows by at least two at a time.
If every gap has size one, apply the next lemma, which shows not
all gaps are filled in this case. In both
of these situations, the fact that a gap
appears in the intervals means that
some $b_n$ is less than
$a_n+b_{n-1}-a_{n-1}-f(a_n)$.
\begin{align}
b_n < a_n + b_{n-1} - a_{n-1} - f(a_n).
\end{align}
If $n$ is the least value for which a gap in the intervals
is left unfilled, then the lower $P$-positions are segments
of the complementary Beatty sequences parameterized by $\alpha$.
The inequality above implies
\begin{align}
b_n &\leq a_n + b_{n-1} - a_{n-1} - f(a_n)\\
&= a_n + b_{n-1} - a_{n-1} -(([n\beta]-a_n)-(b_{n-1}-a_{n-1}))
\end{align}
The Beatty element $[n\beta]$ can be written $[n\beta]=b_n+d$
where $d>0$,
since $a_n+b_{n-1}-a_{n-1}+f(a_n)=[n\beta]$. (We have assumed $n$ is
the least value for which a gap appears.) Substituting into the
previous inequality shows
\begin{align}
b_n &\leq a_n + b_{n-1} - a_{n-1} - b_n-d+a_n+b_{n-1}-a_{n-1}\\
2b_n - 2a_n &\leq 2(b_{n-1}-a_{n-1})-d\\
2(b_n-a_n)-2(b_{n-1}-a_{n-1}) &\leq -d.
\end{align}
However, the study of the difference function as in
Lemma \ref{h_function_lemma} reveals that this quantity
is non-negative for complementary Beatty sequence $\{a_k\}$
and $\{b_k\}$. Therefore the sequence are not complementary
Beatty sequences.
\end{proof}

\begin{lemma}[Gap lemma]
Let $f$ be a Beatty constraint function satisfying the property that
any gap between the intervals
\begin{align}
I_{k-1} &= [a_n \!+\! b_{k-1}\!-\!a_{k-1} - f(a_n) \!+\! 1,
		a_n \!+\! b_{k-1}\!-\!a_{k-1} + f(a_n)\!-\!1]\\
I_k &= [a_n \!+\! b_k\!-\!a_k - f(a_n) \!+\! 1,
		a_n \!+\! b_k\!-\!a_k + f(a_n)\!-\!1]
\end{align}
has size $f(a_k)-2f(a_n)+1$ exactly equal to one. Then there exists
$n>1$ and $0<k<n$ such that the gap at $I_{k-1}$ is left unfilled by
any $b_j$. In symbols:
\begin{align}
a_n + b_{k-1}-a_{k-1}+f(a_n) \notin \{b_j\}_{j<n}.
\end{align}
\end{lemma}
\begin{proof}
Suppose conversely that for every $n>1$ and $0<k<n$
satisfying $f(a_k)-2f(a_n)+1=1$, there exists
$b_j$ such that $a_n+b_{k-1}-a_{k-1}+f(a_n)=b_j$.
Then for any $n>1$, the cardinality of the set
\begin{align}
G = \{ 0\!<\! k \!<\! n \;|\; 2f(a_n)=f(a_k)\}
\end{align}
equals that of the set
\begin{align}
H \!=\! \{ 0\!<\! k \!<\! n \;|\; 2f(a_n)\!=\!f(a_k) \;\text{and}\;
		a_n \!+\! b_{k-1}\!-\!a_{k-1}\!+\!f(a_n)\!=\!b_j\}.
\end{align}

By Lemma \ref{trivialnogaplemma}, we can assume $[\beta]=2,3,4$. Recall $f$ achieves the
the following values:
\begin{align}
f(a_n) = 
\begin{cases}
[\beta] & n-1 \in X^c\cap Y\\
[\beta]-2 & n-1 \in X\cap Y^c\\
[\beta]-1 &\text{otherwise}
\end{cases}
\end{align}
where $X=\{[n/\{\alpha\}]\}$ and $Y=\{[n/\{\beta\}]\}$.
By performing casework, if
$2f(a)=f(a')$ and $f\neq 0$, then one of the following is true:
\begin{align}
\begin{cases}
2[\beta] = [\beta]-1 &\implies \beta<0 \\
2[\beta] = [\beta]-2 &\implies \beta < 0\\
2([\beta]-1) = [\beta] &\implies [\beta]=2\\
2([\beta]-1) = [\beta]-2 &\implies [\beta]=0\\
2([\beta]-2) = [\beta] &\implies [\beta]=4\\
2([\beta]-2) = [\beta]-1 &\implies [\beta]=3.
\end{cases}
\end{align}
This shows that we need only consider the cases
$[\beta]=2,3,4$. It also shows for a given Beatty constraint function $f$
that the equality $2f(a)=f(a')$ is achieved for a unique pair of values
$f(a)$ and $f(a')$. Assume $[\beta]=2$.
The casework showed we require $f(a_n)=1$ and $f(a_k)=2$.
The value $f(a_k)=2$ is achieved if and only if $k-1$ lies in the
sequence $X^c\cap Y$. Note that this sequence has density
$(1-\{\alpha\})\{\beta\}$.
The sequence $\{b_{k-1}-a_{k-1}\}_{\geq 1}$ has density $1/(\beta-\alpha)$,
which means for fixed $n$ that the sequence
$\{b_{k-1}-a_{k-1}+a_n+f(a_n)\}_{k\geq 1}$ has the same density.
After taking into account the density of the sequence $\{b_n\}$, one
deduces
\begin{align}
\lim_{\stackrel{n\to\infty}{f(a_n)=1}}\frac{|H|}{n-1}
&\leq (1-\{\alpha\})\{\beta\}/((\beta-\alpha)\beta)\\
&< (1-\{\alpha\})\{\beta\} =
	\lim_{\stackrel{n\to\infty}{f(a_n)=1}}\frac{|G|}{n-1}
\end{align}
which indicates $|H|<|G|$ for sufficiently large $n$.

For $[\beta]=4$, we require $f(a_n)=2$ and $f(a_k)=4$,
so that $k-1\in X^c\cap Y$, which permits the same proof.
For $[\beta]=3$, we require $f(a_n)=2$ and $f(a_k)=1$ so that
$k-1\in X\cap Y^c$. To adapt the argument in this case,
replace the density $(1-\{\alpha\})\{\beta\}$ with
$\{\alpha\}(1-\{\beta\})$ after considering the pre-images
of each value $f$ achieves.

Because $|H|<|G|$, for sufficiently large $n$ in any case, we know that the
sets $H$ and $G$ are not equal. Therefore, there is a gap of size one between
two of the intervals which is left unfilled by any $b_j$.
\end{proof}

The next lemma demonstrates that Beatty $P$-positions
necessarily arise from Beatty constraint functions; i.e.,
the inverse problem has a unique solution which, if it exists, coincides with
a candidate function.
\begin{lemma}\label{injectivity}
Suppose a game in Definition \ref{modifiedgame} is played with the
constraint function $f$ and moreover that the $P$-positions agree with pairs
of the complementary Beatty sequences $\{a_n\}$
and $\{b_n\}$ parameterized by $\alpha$.
Then $f$ agrees with the function $\widetilde{f}$ in Equation
\eqref{modified_beatty_constraint_function} over all $a_n$ with $n$
positive.
\end{lemma}
\begin{proof}
Let $1<\alpha<2<\beta$ satisfy $1/\alpha+1/\beta=1$.
Denote the $P$-positions by
\begin{align}
\{(a_n,b_n)\} = \{([n\alpha], [n\beta])\}.
\end{align}
To proceed by induction, suppose
$f(a_n)=\widetilde{f}(a_n)$ for all $n<N$. The $N$th
$P$-position is found be evaluating
\begin{align}
a_N &= \mex\{a_k,b_k\}_{k<N}\\
b_N &= \mex_{\geq a_N}\{b_k\}_{k<N}
    \cup[a_N + b_k\!-\!a_k - f(a_N)\!+\!1,
    a_N + b_k\!-\!a_k + f(a_N)\!-\!1]_{k<N}.
\end{align}
If there is no gap in the intervals, then the mex can be computed
by adding one to the greatest interval's right-endpoint, to obtain
\begin{align}
b_N = a_N + b_{N-1}-a_{N-1}+f(a_N).
\end{align}
This implies $f(a_N) = (b_N-a_N)-(b_{N-1}-a_{N-1})=\widetilde{f}(a_N)$,
which is the desired result.
Otherwise, there is a gap between the intervals $I_k$ and $I_{k+1}$, which implies that
\begin{align}
b_N = a_N + b_k-a_k + f(a_N)
\end{align}
and
\begin{align}
f(a_N) = (b_N-a_N)-(b_k-a_k) = \sum_{j=k+1}^N \widetilde{f}(a_j)
\geq f(a_{k+1}).
\end{align}
But the existence of a gap implies that $2f(a_N)-f(a_{k+1})\leq 0$
or that $f(a_N)\leq f(a_{k+1})/2$. The
joint inequality $f(a_{k+1})\leq f(a_N)
\leq f(a_{k+1})/2$ implies
$f(a_{k+1})=f(a_N)=0$, which contradicts the definition of
$f$ (recall that if $f=0$ at any point, then the $P$-positions are not complementary).
Therefore, there is no gap in the intervals, which implies
$f(a_n)=\widetilde{f}(a_n)$ for all positive $n$.
\end{proof}

The next theorem extends the discussion preceding Lemma \ref{h_function_lemma}
to games in the family of Definition \ref{modifiedgame} and determines the range of
possible Beatty sequences for the forward problem. Figure
\ref{modifiedbeattyirrationals} illustrates the set of permitted $\alpha$.
\begin{theorem}[Unifying Theorem]\label{beattymodified2pilesaregolden}
Let a game in Definition \ref{modifiedgame} have $P$-positions which are
the complementary Beatty sequences parameterized by $\alpha$. Then
one of the following is true:
\begin{enumerate}
\item[(i)] $\alpha=\goldenformula$ for some integer $t>0$,
\item[(ii)] $[\beta]=3$ or $[\beta]=4$ and there exist integers,
$p,q>0$ such that
\begin{align}
\alpha = \frac{\sqrt{4pq \!+\! ([\beta]p\!-\!1)^2}
	+ 2q \!-\! ([\beta]p \!-\! 1)}{2q}, \label{34family}
\end{align}
\item[(iii)] $[\beta]=4$ and there exist integers $p,q>0$ such that
\begin{align}
\alpha = \frac{\sqrt{4pq + (q \!-\! 3p \!-\! 1)^2} + 3q \!-\! 3p \!-\! 1}{2q}, \label{4family}
\end{align}
\item[(iv)] $[\beta]\geq 5$ or, equivalently, $\alpha<5/4$.
\end{enumerate}
\end{theorem}
\begin{proof}
If $f\equiv t$ for some $t>0$, then $\alpha$ is as in case (i) by Lemma
\ref{h_function_lemma}. Otherwise, the previous lemmata showed we require
that $f$ is a Beatty constraint function as in Definition
\ref{modified_beatty_constraint_function} and that $2\min f -\max f\geq 1$.

To proceed, we 
inspect the possible values of $f$ as $\alpha$ values,
some of which indicate this inequality  is true.
Recall the trichotomy Equation \eqref{trichotomy}, which says that if 
$X=\{[n/\{\alpha\}]\}$ and $Y=\{[n/\{\beta\}]\}$, then
\begin{align}
d_\alpha(n) = \begin{cases}
0 & n \in (X\cap Y)\cup (X^c\cap Y^c)\\
1 & n\in X^c\cap Y\\
-1 & n\in X\cap Y^c
\end{cases} \!\implies\! f = \begin{cases}
[\beta]-1\\
[\beta]\\
[\beta]-2.
\end{cases} \label{trichotomy_f}
\end{align}
If $f$ achieves more than one value, then $f$ achieves each of those values 
infinitely many times as discussed in Lemma \ref{h_function_lemma}. Therefore
we can determine the minimum and maximum by looking at the three values in
Equation \eqref{trichotomy_f}.

Suppose now that $f$ is three-valued. We then require
\begin{align}
2([\beta]-2)-[\beta] \geq 1 \iff [\beta]\geq 5.
\end{align}
This case was handled in Lemma \ref{trivialnogaplemma}.
Therefore suppose $f$ achieves only two values.
If those two values are $[\beta]-2$ and $[\beta]$,
then we require the same inequality, which, again, was already handled.
If $f$ achieves only the values $[\beta]-2$ and $[\beta]-1$, then we require
\begin{align}
2([\beta]-2)-([\beta]-1) \geq 1 \iff [\beta]\geq 4.
\end{align}
If $f$ achieves only the values $[\beta]-1$ and $[\beta]$, then we require
\begin{align}
2([\beta]-1)-[\beta] \geq 1 \iff [\beta]\geq 3.
\end{align}
Therefore, if $[\beta]=3$ or $[\beta]=4$, then the inequality
$2\min f - \max f \geq 1$ is satisfied only when $f$ excludes one or
more of its possible values from its range.

To handle these cases in an explicit fashion, start by letting $[\beta]=3$ or 
$[\beta]=4$ and exclude the value $f=[\beta]-2$ by setting $X\cap 
Y^c=\emptyset$. Theorem 3.13 of \cite{niven} states this is equivalent to the 
existence of integers $p,q>0$ solving $p(1-\{\beta\})+q\{\alpha\}=1$. Note that
\begin{align}
\frac{1}{\alpha} + \frac{1}{\beta} = 1
\end{align}
implies $\beta=\alpha/(\alpha-1)=(1+\{\alpha\})/\{\alpha\}$. Therefore
\begin{align}
\{\beta\} &= (1+\{\alpha\})/\{\alpha\}-[\beta] \quad\text{and}\\
1 - \{\beta\} &= [\beta] - \frac{1}{\{\alpha\}}.
\end{align}
Manipulating the equation $p\left([\beta] - 1/\{\alpha\}\right)
	+ q\{\alpha\} = 1$, one arrives at the quadratic equation
\begin{align}
q\alpha^2 + ([\beta]p \!-\! 1 \!-\! 2q)\alpha + q\!-\!p
	\!+\! ([\beta]p\!-\!1) = 0. \label{34polynomial}
\end{align}
The quadratic formula yields a solution in the interval $(1,2)$
\begin{align}
\alpha = \frac{\sqrt{4pq + ([\beta]p\!-\!1)^2} + 2q - ([\beta]p \!-\! 1)}{2q}.
\end{align}

Finally, we let $[\beta]=4$ and exclude the value $f=[\beta]$ by
setting $X^c\cap Y=\emptyset$. Equivalently, there exist integers $p,q>0$
such that
\begin{align}
p\{\beta\} + q(1-\{\alpha\}) = 1.
\end{align}
Similar to the previous derivation, the quadratic equation
\begin{align}
q\alpha^2 + (3p-3q+1)\alpha + (2q-4p-1)=0 \label{4polynomial}
\end{align}
arises and the quadratic formula yields the solution
\begin{align}
\alpha = \frac{\sqrt{4pq + (q \!-\! 3p \!-\! 1)^2} + 3q \!-\! 3p \!-\! 1}{2q}.
\end{align}
In each of the above cases $2\min f - \max f \geq 2$, and
Theorem \ref{mastertheorem_expanded} implies the equality
$b_n = a_n + b_{n-1}-a_{n-1} + f(a_n)$ is always true.

The preceding casework is exhaustive, yielding the statement.
Figure \ref{modifiedbeattyirrationals} displays visually the
family of irrationals derived.
\end{proof}

\begin{figure}[t]
\begin{tikzpicture}
\begin{axis}[
xtick style={draw=none},
xtick={1, 5/4, 4/3, sqrt(2), (1+sqrt(5))/2, 2},
xmin=1,xmax=2,
ymin=0,ymax=1,
hide y axis,
axis x line*=bottom,
height=0.75*\textwidth/3]
\input{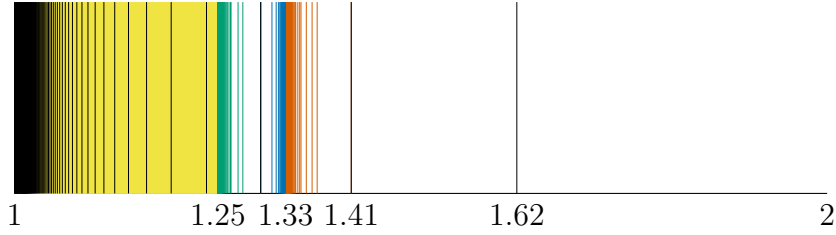}
\end{axis}
\end{tikzpicture}
\caption{Black: irrationals in case (i) of Theorem
\ref{beattymodified2pilesaregolden}.
Red: irrationals in case (ii) with $[\beta]=3$. Green:
irrationals in case (ii) with $[\beta]=4$. Blue:
irrationals in case (iii). Yellow: irrationals in case (iv).}
\label{modifiedbeattyirrationals}
\end{figure}

If $p=q=1$ and $t=[\beta]-1$ are
substituted into Equation 
\eqref{34family}, then the special case of the formula in Equation
\eqref{goldenfamily} appears. Moreover,
setting $q=p$, $m=[\beta]p-1$, and $[\beta]=3,4$
in Equation \eqref{34family}, one can deduce a family of
Definition \ref{modifiedgame} games whose
$P$-positions are pairs of the complementary Beatty sequences
parameterized by
\begin{align}
\alpha &= \frac{\sqrt{m^2 \!+\! 4p^2}
	\!+\! 2p \!-\! m}{2p}.
\end{align}
This is precisely the irrational family mentioned in Section 1.6 of
\cite{larsson_restrictions} and the family of Section 1.4.5 in
\cite{larsson2pile}.
In those papers, the family of irrationals is made in reference to a
game whose $P$-positions are ``close'' to pairs of the complementary
Beatty sequences parameterized by $\alpha$.
The preceding theorem shows that if $m=3p-1$ or $m=4p-1$, then
the $P$-positions may be
brought even closer (to a distance of zero) by playing a game in Definition
\ref{modifiedgame} with the Beatty constraint function parameterized
by $\alpha$. This is in contrast with the statement of the appendix in the
same paper, which states that for $p>1$, there exists no game of
a certain type with $P$-positions which are the complementary Beatty pairs.
The appendix also makes reference to a special family of irrationals
\begin{align}
r = \frac{\sqrt{4p^2+1}+2p-1}{2p}.
\end{align}
This family can be deduced by the same method the family Equation
\eqref{4family} was derived by instead setting $[\beta]=2$ and
letting $q=p$.

In \cite{larssonweimer}, an invariant family of games is defined
whose $P$-positions arise as pairs of the complementary Beatty sequences
parameterized by
\begin{align}
\alpha_k = [1; k, 1, k, 1, k, \dots]
= \frac{1+\sqrt{1+4/k}}{2} = \frac{\sqrt{4k+k^2}+k}{2k}
\end{align}
for all $k\geq 1$.
For any $[\beta_k]\geq 2$, we can substitute
$p=1$ and $q=[\beta_k]-1=k$ into Equation \eqref{34family} to arrive at the
same formula.

It is for the previous examples
that we view the family of quadratic irrationals in the previous theorem
statement as a kind of unifying
family. A special property is that $\Delta^2$ is two valued for these $\alpha$.

The key to our solution to Fraenkel's inverse problem
is to observe that the disagreement of $b_n$ 
with the recurrence formula depends (in a sufficient way)
on the intervals $I_k$ mentioned in the
proof 
of Theorem \ref{mastertheorem_expanded} having gaps between them. To instead be
sure that the intervals always overlap, we can replace each interval's
lower endpoint with $-\infty$. This has the advantage of also
removing the problematic absolute value symbol which yields non-complementary
$P$-positions for vanishing constraint functions.
The next game family is introduced as Definition 
\ref{relaxedgame}. As a corollary, we deduce a family of
games whose $P$-positions are any desired complementary Beatty sequences.

\section{Relaxed Wythoff}
In this section, a third, novel family of games is defined. This family
accepts a constraint function and has three possible move types.
It is shown that some game in this family has $P$-positions which
are pairs of any desired complementary Beatty sequences.
\begin{definition}[Relaxed Wythoff]\label{relaxedgame}
Given two piles of tokens $(x,y)$ of sizes $x,y$, with $0\leq x\leq y<\infty$. 
Two players alternate removing tokens from the piles in one of the
following ways:
\begin{enumerate}
\item[(aa)] Remove any positive number of tokens from a single pile,
    possibly the entire pile.
\item[(bb)] Remove more tokens from the smaller pile than from the larger
	pile. (A special case of the next move type.)
\item[(cc)] Remove a positive number of tokens from each pile by sending
    $(x_0,y_0)$ to $(x_1,y_1)$ where
\begin{align}
(y_0-y_1)-(x_0-x_1) < f(x_0),
\end{align}
\end{enumerate}
where the constraint function $f$ is integer-valued and
non-negative.
\end{definition}

Lemma \ref{doublemex} provides a strategy to prove the following theorem,
which is analogous to Theorem \ref{mastertheorem}
\begin{theorem}[$P$-positions of Relaxed Wythoff]\label{relaxed_solutions}
Let $\mathcal{S}=\{(a_n,b_n)\}_{n=0}^\infty$ where $a_0=b_0=0$ and for all
$n\in\mathbb{Z}_{>0}$ set
\begin{align}
a_n &= \operatorname{mex}\{a_k, b_k \;|\; k < n\}\\
b_n &= f(a_n) + b_{n-1} + a_n - a_{n-1}.
\end{align}
If $f\geq 0$ and $f(1)\geq 1$, then $\mathcal{S} = \{(a_n,b_n)\}_{n\geq 0}$ 
is the set of $P$-positions of a game in Definition \ref{relaxedgame} played 
with the constraint function $f$. Moreover, the sets $\{a_n\}$ and $\{b_n\}$ 
are monotone and complementary sequences of integers covering the natural 
numbers and $\{b_n-a_n\}$ is a monotone sequence.
\end{theorem}
\begin{proof}
We proceed by inductively constructing the $P$-positions similar to the proof 
of Lemma \ref{doublemex}. Let $f$ be as given. The trivial $P$-position is 
$(0,0)$. Similar to Lemma \ref{doublemex}, the first non-trivial $P$-position has
left-component $a_1=\mex\{0\}=1$ and
\begin{align}
b_1 = \min\{b\geq 1 \;|\; (b-b_0)-(1-a_0)
	\geq f(1)\}
\end{align}
where we have omitted the absolute value signs. Note that any $b$ in this set
has the property that
\begin{align}
b \geq 1 + b_0 - a_0 + f(1) = 1 + f(1) = a_1 + f(1).
\end{align}
The least such $b$ equals $b_1=1+f(1)$. By the assumption on $f$, this implies 
$b_1\geq 2$. Note that $b_1-a_1 = f(1) > 0 = b_0-a_0$. Therefore, an induction 
hypothesis that
\begin{align}
\begin{cases}
b_k = a_k + b_{k-1} - a_{k-1} + f(a_k) \quad\text{and}\\
b_k - a_k > b_{k-1} - a_{k-1}
\end{cases} \label{induction_hypothesis}
\end{align}
is true for all $1<k<n$.

For an inductive step, we proceed similar to the derivation
of the first non-trivial $P$-position. After adding a quantifier and
restricting $b\neq b_k$, we obtain
\begin{align}
a_n &= \operatorname{mex}\{a_k, b_k \;|\; k < n\}\\
b_n &= \min\{b \geq a_n \;|\; b\!\neq\! b_k \;\text{and}\;
	(b \!-\! b_k) \!-\! (a_n \!-\! a_k) \!\geq\! f(a_n) \forall k\!<\!n\}.
\end{align}
If $b\neq b_k$ lies in the set restriction, then rearranging shows
\begin{align}
b &\geq a_n+f(a_n)+b_k-a_k \quad\forall k < n.
\end{align}
Since the inequality holds for all $k$ and $\{b_k-a_k\}$ is increasing, we may 
as well set $k=n-1$ and find the minimal such $b$ at the endpoint of the 
inequality:
\begin{align}
b_n = a_n+f(a_n)+b_{n-1}-a_{n-1}, \label{b_solved}
\end{align}
completing the proof of the recurrence formula. Re-arranging Equation 
\eqref{b_solved} reveals
\begin{align}
b_n = b_{n-1} + (a_n-a_{n-1}) + f(a_n),
\end{align}
from which we can deduce that $\{b_n\}$ is strictly monotone, because $\{a_n\}$ 
is strictly monotone and $f\geq 0$. A similar re-arrangement shows
$\{b_n-a_n\}$ is monotone.
For a final remark on the $P$-positions, 
notice that $\{a_n,b_n \in \mathcal{S}\}=\mathbb{Z}_{\geq 0}$ since the first 
component consists of minimum excludants.

To prove that a player in an $N$-position may always move to a $P$-position,
apply the same argument of Lemma \ref{doublemex} with the substitution
of the minimum
\begin{align}
\min\{b \geq a_n \;|\; b\!\neq\! b_k \;\text{and}\;
	(b \!-\! b_k) \!-\! (a_n \!-\! a_k) \geq f(a_n) \forall k<n\}
\end{align}
instead of
\begin{align}
\min\{b \geq a_n \;|\; b\!\neq\!b_k \;\text{and}\;
    |(b\!-\!b_k)\!-\!(a_n\!\!-\!a_k)|  \geq f(a_n)\forall k<n
    \}.
\end{align}
\end{proof}

Re-using some of the results from the proof of Theorem 
\ref{beattymodified2pilesaregolden}, we obtain the following corollary to 
Theorem \ref{relaxed_solutions}.
\begin{corollary}\label{main_result}
If $\{a_n\}=\{[n\alpha]\}$ and $\{b_n\}=\{[n\beta]\}$ are complementary Beatty 
sequences with $\alpha<\beta$, then there is a function $f$ such that if a
game 
in Definition \ref{relaxedgame} is played with $f$ as its constraint function,
then the $P$-positions are precisely the pairs $\{(a_n,b_n)\}_{n=0}^\infty$.
\end{corollary}
\begin{proof}
Theorem \ref{relaxed_solutions} shows that the game in Definition 
\ref{relaxedgame} played with the Beatty constraint function $f$
parameterized by $\alpha$
should yield the desired Beatty $P$-positions. It remains only to prove
 $f(1)\neq 0$ and that $f(a_n)\geq 0$ for all $a_n>0$ to apply
Theorem \ref{relaxed_solutions}.

First,
\begin{align}
f(1) = [\beta]-[\alpha] = [\beta]-1 \geq 2 - 1 = 1,
\end{align}
so the first criterion is satisfied. For the second criterion, recall the proof 
of Theorem \ref{beattymodified2pilesaregolden}, which showed that 
$[\beta]-2\leq f \leq [\beta]$. Since $[\beta]\geq 2$, we can deduce that 
$f\geq 0$. Therefore, the hypotheses on $f$ are satisfied, so the $P$-positions 
of the game in Definition \ref{relaxedgame} with the function $f$ may now be 
evaluated using the recurrence in the statement of this theorem.
It is seen directly that the $P$-positions are the complementary Beatty
sequences parameterized by $\alpha$.
\end{proof}
Therefore, Relaxed Wythoff gives simple rules for which the $P$-positions
are pairs of given complementary Beatty sequences.
In Figure \ref{successful_game}, we return to the example in Figure
\ref{figure1}.
Playing the game as in Definition \ref{relaxedgame} with the same constraint
function yields $P$-positions which coincide with the associated Beatty pairs.
\begin{figure}[t]
{\small
\begin{tabular}{c|c|c|c|c|c|c|c|c|c|c}
$n$  &  0 & 1 & 2 & 3 & 4 & 5 & 6 & 7 & 8 & 9\\ \hline
$a_n$ & 0 & 1 & 2 & 4 & 5 & 7 & 8 & 10 & 11 & 13\\
$b_n$ & 0 & 3 & 6 & 9 & 12 & 16 & 19 & 22 & 25 & 29\\ \hline
$[n\alpha]$ & 0 & 1 & 2 & 4 & 5 & 7 & 8 & 10 & 11 & 13\\
$[n\beta]$ & 0 & 3 & 6 & 9 & 12 & 16 & 19 & 22 & 25 & 29
\end{tabular}
}
\caption{Top: $P$-positions of the game played by Definition \ref{relaxedgame}
with the constraint function in Equation
\eqref{modified_beatty_constraint_function} parameterized
by $\alpha=1+\sqrt{5}/5$. Bottom: Beatty sequences parameterized by $\alpha$.
Compare to Figure \ref{figure1}.}
\label{successful_game}
\end{figure}

\section{Discussion and Conclusion}
The forward problem of finding the $P$-positions of a given game has
existed since game theory began. The inverse problem
of finding a game for given $P$-positions has recently been developed
and offers a fresh perspective on classical problems.
In this paper, we precisely characterized the Beatty
sequences arising from the $P$-positions of a known game family.
The analysis of the sequence $\Delta^2_{n-1}$ was
crucial and yielded insight to finding simple rules for another game with Beatty
$P$-positions. A solution was proposed to the inverse problem posed by
Fraenkel at the 2011 BIRS workshop via the definition of Relaxed Wythoff played
with a Beatty constraint function. Note that this inverse problem
was further refined at the end of \cite{fraenkelinvariance}:
\begin{quote}
Is it possible to find short 2-invariant game rules, without disclosing
any part of the $P$-positions?
\end{quote}
Our proof of Theorem \ref{beattymodified2pilesaregolden} showed that
if a game is played with the Beatty constraint function parameterized
by some $\alpha$ in the statement of the theorem, then the game is $k$-invariant
for some $k\leq 3$. Each of the irrationals in the countable family of
Theorem \ref{mastertheorem_expanded} corresponds to a Beatty constraint function
inducing a 2-invariant game.
Supplying the players with the formula for the Beatty
constraint function automatically
reveals the irrationals which parameterize the Beatty $P$-positions.
Revealing some aspect of the irrationals is a common feature of the games
in \cite{fraenkelinvariance} and \cite{goldberg_rulesets}. In
spite of this, each game is indeed simple. A move can be validated with finite calculations. Moreover, the study of the game family in Definition
\ref{modifiedgame} over Beatty constraint functions
is significant on its own because it expands the study of non-monotone
constraint functions, also seen in \cite{fraenkelboole}. A striking feature
of this analysis revealed that the \emph{uncountable} family of complementary
Beatty sequences parameterized by $\alpha\in(1,5/4)$ arise as $P$-positions
of games in the family of Definition \ref{modifiedgame}.

\section{Acknowledgments}

This work was completed under the care and supervision of several 
institutions and advisors. Gratefully, we acknowledge Hampshire College 
and The Five College Consortium for providing the authors with a place 
to meet each other and explore mathematics in an experimental fashion while the first author was an undergraduate student. 
Amherst College and later Smith College provided new academic homes for the second
author after leaving Hampshire College. Once the first
author completed his studies at 
Hampshire College, he was given a host of new ideas at The University 
of California, Santa Cruz, particularly from Dr. Fran\c{c}ois Monard.
The first author also thanks The Ross Program and Dr. Jim Fowler for 
providing a motivating and hospitable work environment. The authors also thank
a reviewer for helpful references and suggestions.



\end{document}